\documentclass[12pt, a4paper]{article}

\usepackage[latin2]{inputenc}
\usepackage{amssymb}
\usepackage{paralist}
\usepackage{amsmath, calc}
\usepackage{amsfonts}
\usepackage{amsthm}
\usepackage{fullpage}
\usepackage{enumerate}
\usepackage{paralist}

\newtheorem{theorem}{Theorem}[section] %
\newtheorem{lemma}[theorem]{Lemma} %

\usepackage{cite}

\begin{document}

\title{Unique representations of integers by linear forms}

\author{
S\'andor Z. Kiss \thanks{Department of Algebra, Institute of Mathematics, Budapest University of
Technology and Economics, M\H{u}egyetem rkp. 3., H-1111, Budapest, Hungary. Email: ksandor@math.bme.hu.
This author was supported by the NKFIH Grant No. K129335.}, Csaba
S\'andor \thanks{Department of Stochastics, Institute of Mathematics, Budapest University of
Technology and Economics, M\H{u}egyetem rkp. 3., H-1111, Budapest, Hungary. Department of Computer Science and Information Theory, Budapest University of Technology and Economics, M\H{u}egyetem rkp. 3., H-1111 Budapest, Hungary, MTA-BME Lend\"ulet Arithmetic Combinatorics Research Group,
  ELKH, M\H{u}egyetem rkp. 3., H-1111 Budapest, Hungary . Email: csandor@math.bme.hu.
This author was supported by the NKFIH Grant No. K129335.}
}

\date{}

\maketitle

\begin{abstract}
\noindent Let $k\ge 2$ be an integer and let $A$ be a set of nonnegative integers. For a $k$-tuple of positive integers $\underline{\lambda} = (\lambda_{1}, \dots{} ,\lambda_{k})$ with $1 \le \lambda_{1} < \lambda_{2} < \dots{} < \lambda_{k}$, we define the additive representation function $R_{A,\underline{\lambda}}(n) 
= |\{(a_{1}, \dots{} ,a_{k})\in A^{k}: \lambda_{1}a_{1} + \dots{} + \lambda_{k}a_{k} = n\}|$. 
For $k = 2$, Moser constructed a set $A$ of nonnegative integers such that $R_{A,\underline{\lambda}}(n) = 1$
holds for every nonnegative integer $n$. In this paper we characterize all the $k$-tuples $\underline{\lambda}$ and the sets $A$ of nonnegative integers with $R_{A,\underline{\lambda}}(n) = 1$ for every integer $n\ge 0$.


\end{abstract}

\section{Introduction}
Let $\mathbb{N}$ be the set of nonnegative integers, $\mathbb{Z}^{+}$ the set of positive integers. For an integer $k\ge 2$ and a set $A$ of nonnegative integers, let $R_{A,k}(n)$ and $R_{A,k}^{\le}(n)$ be the number of solutions of the equations 
\[
a_{1} + \dots{} + a_{k} = n, \hspace*{5mm} a_{1}, \dots{} ,a_{k}\in A,
\]
\[
a_{1} + \dots{} + a_{k} = n, \hspace*{5mm} a_{1}, \dots{} ,a_{k}\in A, \hspace*{5mm} a_{1} \le \dots{} \le a_{k}.
\]
Furthermore, consider the $k$-tuples of positive integers $\underline{\lambda} = (\lambda_{1}, \dots{} ,\lambda_{k})$ with $1 \le \lambda_{1} \le \lambda_{2} \le \dots{} \le \lambda_{k}$. The additive representation function associated to linear forms is defined by
\[
R_{A,\underline{\lambda}}(n) = |\{(a_{1}, \dots{} ,a_{k})\in A^{k}: \lambda_{1}a_{1} + \dots{} + \lambda_{k}a_{k} = n\}|,
\]
where $|.|$ denotes the cardinality of a set. 

It is easy to see that the representation function $R_{A,2}(n)$ cannot be constant from a certain point on. On the other hand, Dirac proved \cite{gd} that $R_{A,2}^{\le}(n)$ also cannot be constant by using an analytic argument. In 1962, Moser \cite{lm} constructed a set $A$ of nonnegative integers with $R_{A,\underline{\lambda}}(n) = 1$ for every nonnegative integer $n$, where $k = 2$, $\underline{\lambda} = \{1, \lambda_{2}\}$, $\lambda_{2} \ge 2$.
Later, A. S\'ark\"ozy and V.T. S\'os asked \cite{st} how could be extended the above result to representation functions of linear forms, i.e., for which $k$-tuples $\underline{\lambda}$ there exists an infinite set $A$ of nonnegative integers such that $R_{A,\underline{\lambda}}(n) = 1$ from a certain point on.
Cilleruelo and Ru\'e showed  \cite{cr} that there does not exist an infinite set $A$ of nonnegative integers such that  $R_{A,\underline{\lambda}}(n)$ is a constant for every nonnegative integer $n$ large enough, where $k = 2$, $\underline{\lambda} = (\lambda_{1}, \lambda_{2})$ with $2 \le \lambda_{1} \le \lambda_{2}$. In \cite{jr}, Ru\'e partially answered the above question of A. S\'ark\"ozy and V.T. S\'os by proving that if in the $k$-tuples $\underline{\lambda} = (\lambda_{1}, \dots{} ,\lambda_{k})$ each $\lambda_{i} $ appear exactly $m\ge 2$ times, then there does not exist an infinite set $A$ of nonnegative integers such that 
$R_{A,\underline{\lambda}}(n)$ becomes a constant from a certain point on. On the other hand, Ru\'e and Spiegel showed \cite{rs} that if the $\lambda_{i}$'s are pairwise coprime, then there does not exist such a set as well. More precisely, they proved that if each $\lambda_{i}$ is a product of elements of some set of pairwise coprime positive integers, then such a set $A$ cannot exist.

In this paper we give a complete description of the $k$-tuples $\underline{\lambda}$ and the sets $A$ of nonnegative integers such that $R_{A,\underline{\lambda}}(n) = 1$ for every nonnegative integer $n$.

Let $u_{1}, u_{2}, \dots{} ,v_{1}, v_{2}, \dots{} \in \mathbb{Z}^{+}$ and define $U_{i} = u_{1}u_{2} \cdots{} u_{i}$, $V_{i} = v_{1}v_{2} \cdots{} v_{i}$. Let
\[
S(u_{1}, v_{1}, u_{2}, v_{2}, \dots{} ,u_{k-1}, v_{k-1}, u_{k}) = 
\]
\[
= \{i_{0} + i_{1}U_{1}V_{1} + i_{2}U_{2}V_{2} + \dots{} + i_{k-1}U_{k-1}V_{k-1}: 0\le i_{h} < u_{h+1}, h = 0,1, \dots{} ,k-1\}, 
\]
\[
S(u_{1}, v_{1}, u_{2}, v_{2}, \dots{}) = \bigcup_{k=1}^{\infty}S(u_{1}, v_{1}, u_{2}, v_{2}, \dots{} ,u_{k-1}, v_{k-1}, u_{k}), 
\]
\[
T(u_{1}, v_{1}, u_{2}, v_{2}, \dots{} ,u_{k}, v_{k}) = 
\]
\[
= \{j_{0}U_{1} + j_{1}U_{2}V_{1} + j_{2}U_{3}V_{2} + \dots{} + j_{k-1}U_{k}V_{k-1}: 0\le j_{h} < v_{h+1}, h = 0,1, \dots{} ,k-1\},
\]
\[
T(u_{1}, v_{1}, u_{2}, v_{2}, \dots{}) = \bigcup_{k=1}^{\infty}T(u_{1}, v_{1}, u_{2}, v_{2}, \dots{} ,u_{k-1}, v_{k-1}, u_{k}, v_{k}) 
\]
and
\[
T(u_{1}, v_{1}, u_{2}, v_{2}, \dots{} ,u_{k}, \infty) = 
\]
\[
= \{j_{0}U_{1} + j_{1}U_{2}V_{1} + j_{2}U_{3}V_{2} + \dots{} + j_{k-1}U_{k}V_{k-1}: 0\le j_{h} < v_{h+1}, h = 0,1, \dots{} ,k-2, j_{k-1}\in \mathbb{N}\}.
\]

Let $\lambda \ge 2$ be an integer. Let 
\[
\Lambda_{\lambda}(u_{1}, v_{1}, u_{2}, v_{2}, \dots{} ,u_{k-1}, v_{k-1}, u_{k}) = \{\lambda^{s}: s\in S(u_{1}, v_{1}, u_{2}, v_{2}, \dots{} ,u_{k-1}, v_{k-1}, u_{k})\},
\]
\[
\Lambda_{\lambda}(u_{1}, v_{1}, u_{2}, v_{2}, \dots{}) = \{\lambda^{s}: s\in S(u_{1}, v_{1}, u_{2}, v_{2}, \dots{})\},
\]
\[
A_{\lambda}(u_{1}, v_{1}, u_{2}, v_{2}, \dots{} ,u_{k},v_{k}) = \left \{\sum_{t\in T(u_{1}, v_{1}, u_{2}, v_{2}, \dots{} ,u_{k}, v_{k})}\delta_{t}\lambda^{t}: 0\le \delta_{t} < \lambda \right\},
\]
\[
A_{\lambda}(u_{1}, v_{1}, u_{2}, v_{2}, \dots{}) = \left \{\sum_{t\in T(u_{1}, v_{1}, u_{2}, v_{2}, \dots{} )}\delta_{t}\lambda^{t}: 0\le \delta_{t} < \lambda \right\},
\]
and
\[
A_{\lambda}(u_{1}, v_{1}, u_{2}, v_{2}, \dots{} ,u_{k}, \infty) = \left\{\sum_{t\in T(u_{1}, v_{1}, u_{2}, v_{2}, \dots{} ,u_{k}, \infty)}\delta_{t}\lambda^{t}: 0\le \delta_{t} < \lambda \right\}.
\]

\begin{theorem}
Let $\Lambda = \{\lambda_{1}, \dots{} ,\lambda_{m}\}\subset \mathbb{Z}^{+}$ with $1\le \lambda_{1} < \dots{} < \lambda_{m}$ and let
$\underline{\lambda} = (\lambda_{1}, \dots{} ,\lambda_{m})$. Then for a proper subset $A$ of $\mathbb{N}$, one has $R_{A,\underline{\lambda}}(n) = 1$ for every nonnegative integer $n$ if and only if there exists an integer $\lambda \ge 2$ and integers $u_{1}, v_{1}, u_{2}, v_{2}, \dots{} ,u_{k-1}, v_{k-1}, u_{k} \ge 2$ such that 
\[
\Lambda  = \Lambda_{\lambda}(u_{1}, v_{1}, u_{2}, v_{2}, \dots{} ,u_{k-1}, v_{k-1}, u_{k})
\]
and
\[
A = A_{\lambda}(u_{1}, v_{1}, u_{2}, v_{2}, \dots{} ,u_{k-1}, v_{k-1}, u_{k}, \infty).
\]
\end{theorem}

In Theorem 1.1, the set $\Lambda$ is finite. In the following two theorems the set $\Lambda$ is infinite.

\begin{theorem}
Let $\Lambda = \{\lambda_{1}, \lambda_{2}, \dots{}\}\subset \mathbb{Z}^{+}$ be an infinite set with $1\le \lambda_{1} < \lambda_{2} < \dots{}$ and let $A = \{a_{1}, a_{2}, \dots{}\}\subset \mathbb{N}$ with $0\le a_{1} < a_{2} < \dots{}$. For every $n\in \mathbb{N}$, there exists a unique vector of positive integers $(i_{n,1}, i_{n,2}, \dots{})$ such that
\[
n = \sum_{j=1}^{\infty}\lambda_{j}a_{i_{{n,j}}}
\]
if and only if there exist an integer $\lambda \ge 2$ and $u_{1}, v_{1}, u_{2}, v_{2}, \dots{} \ge 2$ integers such that 
\[
\Lambda  = \Lambda_{\lambda}(u_{1}, v_{1}, u_{2}, v_{2}, \dots{})
\]
and
\[
A = A_{\lambda}(u_{1}, v_{1}, u_{2}, v_{2}, \dots{}).
\]
\end{theorem}
Note that Theorem 1.2 can be proved similarly as Theorem 1.1, hence we omit the proof and leave the details to the reader. The next theorem follows immediately.
\begin{theorem}
Let $\Lambda = \{\lambda_{1}, \lambda_{2}, \dots{}\}\subset \mathbb{Z}^{+}$ be an infinite set with $1\le \lambda_{1} < \lambda_{2} < \dots{}$ and let $A = \{a_{1}, a_{2}, \dots{}\}\subset \mathbb{Z}^{+}$ with $1\le a_{1} < a_{2} < \dots{}$. For every $n\in \mathbb{Z}^{+}$, there exist unique finite $k_{n}$-tuples of positive integers $(i_{1,n}, \dots{} ,i_{k_{n},n})$ and $(j_{1,n}, \dots{} ,j_{k_{n},n})$, $1 \le j_{1,n} < j_{2,n} < \dots{} 
 < j_{k_{n},n}$ such that
\[
n = \sum_{h=1}^{k_{n}}\lambda_{j_{h,n}}a_{i_{h,n}}
\]
if and only if there exist $\lambda \ge 2$ and $u_{1}, v_{1}, u_{2}, v_{2}, \dots{} ,u_{k}, v_{k} \dots{} \ge 2$ integers such that 
\[
\Lambda  = \Lambda_{\lambda}(u_{1}, v_{1}, u_{2}, v_{2}, \dots{})
\]
and
\[
A = A_{\lambda}(u_{1}, v_{1}, u_{2}, v_{2}, \dots{})\setminus \{0\}.
\]
\end{theorem}

\section{Proof of Theorem 1}

First we prove the following lemma. 

\begin{lemma}
Let $u_{1}, v_{1}, u_{2}, v_{2}, \dots{} ,u_{k-1}, v_{k-1} \ge 2$, $u_{k} \ge 1$, $\lambda\ge 2$ 
be integers. 
\begin{itemize}
    \item [(i)] If $A = A_{\lambda}(u_{1}, v_{1}, u_{2}, v_{2}, \dots{} ,u_{k-1}, v_{k-1})$ and 
    
    $\Lambda  = \Lambda_{\lambda}(u_{1}, v_{1}, u_{2}, v_{2}, \dots{} ,u_{k-1}, v_{k-1}, u_{k}) = \{\lambda_{1}, \dots{} ,\lambda_{m}\}\subset \mathbb{Z}^{+}$, $\underline{\lambda} = (\lambda_{1}, \dots{} ,\lambda_{m})$ with $1\le \lambda_{1} < \dots{} < \lambda_{m}$, then for every $0 \le n < \lambda^{U_{k}V_{k-1}}$ one has $R_{A,\underline{\lambda}}(n) = 1$, and for every $n \ge  \lambda^{U_{k}V_{k-1}}$ one has $R_{A,\underline{\lambda}}(n) = 0$. 
    \item [(ii)] If $A = A_{\lambda}(u_{1}, v_{1}, u_{2}, v_{2}, \dots{} ,u_{k}, v_{k})$ and
    
    $\Lambda  = \Lambda_{\lambda}(u_{1}, v_{1}, u_{2}, v_{2}, \dots{} ,u_{k-1}, v_{k-1}, u_{k}) = \{\lambda_{1}, \dots{} ,\lambda_{m}\}\subset \mathbb{Z}^{+}$, $\underline{\lambda} = (\lambda_{1}, \dots{} ,\lambda_{m})$ with $1\le \lambda_{1} < \dots{} < \lambda_{m}$, then for every $0 \le n < \lambda^{U_{k}V_{k}}$ one has $R_{A,\underline{\lambda}}(n) = 1$, and for every $n \ge  \lambda^{U_{k}V_{k}}$ one has $R_{A,\underline{\lambda}}(n) = 0$. 
\end{itemize}
\end{lemma}

\begin{proof}
We prove only part (i) because the proof of (ii) is very similar; hence we omit the details. If
\[
S(u_{1}, v_{1}, u_{2}, v_{2}, \dots{} ,u_{k-1}, v_{k-1}, u_{k}) = 
\]
\[
= \{i_{0} + i_{1}u_{1}v_{1} + i_{2}u_{2}u_{2}u_{1}v_{1} + \dots{} + i_{k-1}u_{k-1}v_{k-1}v_{k-2}u_{k-2}\cdots{} v_{1}u_{1}:
\]
\[
0\le i_{h} < u_{h+1}, h = 0,1, \dots{} ,k-1\}, 
\]
\[
T(u_{1}, v_{1}, u_{2}, v_{2}, \dots{} ,u_{k-1}, v_{k-1}) = 
\]
\[
= \{j_{0}u_{1} + j_{1}u_{2}v_{1}u_{1} + j_{2}u_{3}v_{2}u_{2}v_{1}u_{1} + \dots{} + j_{k-2}u_{k-1}v_{k-2}u_{k-2}\cdots{} v_{1}u_{1}: 
\]
\[
0\le j_{h} < v_{h+1}, h = 0,1, \dots{} ,k-2\},
\]
then every $0 \le m < U_{k}V_{k-1}$ can be uniquely written in the form $m = s + t$, where $s\in S$ and $t\in T$. Since every $0 \le n < \lambda^{U_{k}V_{k-1}}$ can be uniquely written in the form 
\[
n = \sum_{i=0}^{U_{k}V_{k-1}}\delta_{i}\lambda^{i},
\]
where $0 \le \delta_{i} < \lambda$ are integers, and so $n$ also can be written uniquely in the form
\[
n = \sum_{s\in S}\sum_{t\in T}\delta_{s+t}\lambda^{s+t} = \sum_{s\in S}\lambda^{s}\sum_{t\in T}\delta_{s+t}\lambda^{t}, 
\]
where $\lambda^{s}\in \Lambda$ and $\sum_{t\in T}\delta_{s+t}\lambda^{t}\in A$. Thus every $0 \le n < \lambda^{U_{k}V_{k-1}}$ can be written in the form 
\[
n = \lambda_{1}a^{(1)} + \dots{} + \lambda_{m}a^{(m)},
\]
where $a^{(1)}, \dots{} ,a^{(m)}\in A$. The number of these numbers is $|A|^{|\Lambda|} = \lambda^{U_{k}V_{k-1}}$, and so if $n \ge \lambda^{U_{k}V_{k-1}}$, then $R_{A,\underline{\lambda}}(n) = 0$.  
\end{proof}

Furthermore, we need the following lemma.

\begin{lemma}
Let $S\subseteq \mathbb{N}$ with $0\in S$ and $\underline{\mu} = (\mu_{1}, \dots{} ,\mu_{r})\subset \mathbb{Z}^{+}$ with $1\le \mu_{1} < \dots{} < \mu_{r}$. If there exist index sets $I,J\subseteq \{1, \dots{} ,r\}$ with $I \neq J$ and $I = \{i_{1}, \dots{} ,i_{k}\}$, $i_{1} < \dots{} < i_{k}$ and $J = \{j_{1}, \dots{} ,j_{l}\}$, $1 \le j_{1} < \dots{} < j_{l}$ 
and there exist non zero integers $s^{(i_{1})}, \dots{} ,s^{(i_{k})}, s^{(j_{1})}, \dots{} ,s^{(j_{l})}\in S$ 
such that 
\[
\mu_{i_{1}}s^{(i_{1})} + \dots{} + \mu_{i_{k}}s^{(i_{k})} = \mu_{j_{1}}s^{(j_{1})} + \dots{} + \mu_{j_{l}}s^{(j_{l})},
\]
then there exist a nonnegative integer $N$ such that $R_{S,\underline{\mu}}(N) \ge 2$.
\end{lemma}

\begin{proof}
Let $I^{c} = [1,r]\setminus I$ and $J^{c} = [1,r]\setminus J$. Since
\[
N = \sum_{i\in I}\mu_{i}s^{(i)} + \sum_{i\in I^{c}}\mu_{i}\cdot 0 = \sum_{j\in J}\mu_{j}s^{(j)} + \sum_{j\in J^{c}}\mu_{i}\cdot 0,
\]
then $R_{S,\underline{\mu}}(N) \ge 2$.
\end{proof}

Now we prove Theorem 1.1. In the first step, we prove the sufficient part, so for $\Lambda = \{\lambda_{1}, \dots{} ,\lambda_{m}\}\subset \mathbb{Z}^{+}$ with $1\le \lambda_{1} < \dots{} < \lambda_{m}$ and 
$\underline{\lambda} = (\lambda_{1}, \dots{} ,\lambda_{m})$, let us suppose that there exists an integer $\lambda \ge 2$ and $u_{1}, v_{1}, u_{2}, v_{2}, \dots{} ,u_{k-1}, v_{k-1}, u_{k} \ge 2$ integers such that 
\[
\Lambda  = \Lambda_{\lambda}(u_{1}, v_{1}, u_{2}, v_{2}, \dots{} ,u_{k-1}, v_{k-1}, u_{k})
\]
and
\[
A = A_{\lambda}(u_{1}, v_{1}, u_{2}, v_{2}, \dots{} ,u_{k-1}, v_{k-1}, u_{k}, \infty).
\] Obviously,
\[
A_{\lambda}(u_{1}, v_{1}, u_{2}, v_{2}, \dots{} ,u_{k}, 1) \subseteq A_{\lambda}(u_{1}, v_{1}, u_{2}, v_{2}, \dots{} ,u_{k}, 2) \subseteq A_{\lambda}(u_{1}, v_{1}, u_{2}, v_{2}, \dots{} ,u_{k}, 3) \subseteq \dots{}
\]
and
\[
A_{\lambda}(u_{1}, v_{1}, u_{2}, v_{2}, \dots{} ,u_{k}, \infty) = \bigcup_{v=1}^{\infty}A_{\lambda}(u_{1}, v_{1}, u_{2}, v_{2}, \dots{} ,u_{k}, v).
\]
By (ii) of Lemma 2.1, for every $v \ge 1$ we have 
\[
R_{A_{\lambda}(u_{1}, v_{1}, u_{2}, v_{2}, \dots{} ,u_{k}, v),\underline{\lambda}}(n) = 1
\]
for every integer $0 \le n < \lambda^{vU_{k}V_{k-1}}$ and so 
\[
R_{A_{\lambda}(u_{1}, v_{1}, u_{2}, v_{2}, \dots{} ,u_{k}, \infty),\underline{\lambda}}(n) = 1
\]
for every $n\in \mathbb{N}$.

Now we prove the necessity part. Let $A = \{a_{1}, a_{2}, \dots{} \}$, $0\le a_{1} < a_{2} \dots{}$ and  
$\underline{\lambda} = (\lambda_{1}, \dots{} ,\lambda_{m})$ with $1\le \lambda_{1} < \dots{} < \lambda_{m}$.  Suppose that $R_{A,\underline{\lambda}}(n) = 1$ for every nonnegative integer $n$. 
Since  $R_{A,\underline{\lambda}}(0) = 1$, we have 
\[
0 = \lambda_{1}a^{(1)} + \dots{} + \lambda_{m}a^{(m)},
\]
with $a^{(1)}, \dots{} ,a^{(m)}\in A$, which implies that $a_{1} = 0$. Since $R_{A,\underline{\lambda}}(1) = 1$, there exists a representation
\[
1 = \lambda_{1}a^{(1)} + \dots{} + \lambda_{m}a^{(m)},
\]
with $a^{(1)}, \dots{} ,a^{(m)}\in A$, which implies that $a_{2} = 1$ and $\lambda_{1} = 1$. Now assume that $a_{i} = i - 1$ for $1 \le i \le \lambda$, but $a_{\lambda+1} \neq \lambda$. We show that $\lambda_{2} = \lambda$. If $\lambda_{2} < \lambda$, then $\lambda_{2} = \lambda_{1}a_{\lambda_{2}+1} = \lambda_{2}a_{2}$. In view of Lemma 2.2, this is a contradiction. If $\lambda_{2} > \lambda$, then in
the following representation
\[
\lambda = \lambda_{1}a^{(1)} + \dots{} + \lambda_{m}a^{(m)},
\]
with $a^{(1)}, \dots{} ,a^{(m)}\in A$, we have $a^{(i)} = 0$ for every $2 \le i \le m$ and so $\lambda = \lambda_{1}a^{(1)} = a^{(1)}$, which implies that $\lambda\in A$, a contradiction. 

Assume that there exists an integer $u_{1} \ge 2$ such that  $\lambda_{1} = 1$, $\lambda_{2} = \lambda$, 
$\lambda_{3} = \lambda^{2}, \dots{} ,\lambda_{u_{1}} = \lambda^{u_{1}-1}$, but $\lambda_{u_{1}+1}\neq \lambda^{u_{1}}$. Now we prove that $\lambda_{u_{1}+1} > \lambda^{u_{1}}$ and $a_{\lambda+1} = \lambda^{u_{1}}$. If $\lambda_{u_{1}+1} < \lambda^{u_{1}}$, then $\lambda_{u_{1}+1} = \lambda_{u_{1}+1}a_{2}$,
but the representation 
\[
\lambda_{u_{1}+1} = \delta_{0} + \delta_{1}\lambda + \dots{} + \delta_{u_{1}-1}\lambda^{u_{1}-1},
\]
with $0 \le \delta_{i} < \lambda$ gives the representation
\[
\lambda_{u_{1}+1} = \lambda_{1}a_{\delta_{0}+1} + \lambda_{2}a_{\delta_{1}+1} + \dots{} + \lambda_{u_{1}}a_{\delta_{u_{1}-1}+1}.
\]
By Lemma 2.2, this is a contradiction. Furthermore, if $a_{\lambda+1} < \lambda^{u_{1}}$, then by using $a_{\lambda+1} = \lambda_{1}a_{\lambda+1}$, the representation
\[
a_{\lambda+1} = \delta_{0} + \delta_{1}\lambda + \dots{} + \delta_{u_{1}-1}\lambda^{u_{1}-1},
\]
with $0 \le \delta_{i} < \lambda$ gives the representation
\[
a_{\lambda+1} = \lambda_{1}a_{\delta_{0}+1} + \lambda_{2}a_{\delta_{1}+1} + \dots{} + \lambda_{u_{1}}a_{\delta_{u_{1}-1}+1}.
\]
In view of Lemma 2.2, this is a contradiction. If $a_{\lambda+1} > \lambda^{u_{1}}$, then
in the following representation
\[
\lambda^{u_{1}} = \lambda_{1}a^{(1)} + \dots{} + \lambda_{m}a^{(m)},
\]
with $a^{(i)}\in A$, by $\lambda_{u_{1}+1} > \lambda^{u_{1}}$, we have $a^{(i)} = 0$ for every $i > u_{1}$. Moreover, for every $1\le i \le u_{1}$, then $a^{(i)}\le \lambda^{u_{1}}$, and so $0 \le a^{(i)}< \lambda$. 
\[
\lambda^{u_{1}} = \lambda_{1}a^{(1)} + \dots{} + \lambda_{u_{1}}a^{u_{1}} \le 1\cdot(\lambda-1)+\lambda^{2}(\lambda-1)+\lambda^{2}(\lambda-1)+\dots{} +\lambda^{u_{1}-1}(\lambda-1) 
\]
\[
= \lambda^{u_{1}} - 1,
\]
which is a contradiction and so  $a_{\lambda+1} = \lambda^{u_{1}}$.

Next, we show that for some $v\in \mathbb{Z}^{+}$, if $\Lambda \cap [0,\lambda^{vu_{1}} - 1] = \Lambda_{\lambda}(u_{1})$, $A \cap [0,\lambda^{vu_{1}} - 1] = A(u_{1},v)$ with $\lambda^{vu_{1}}\in A$, then
$\Lambda \cap [0,\lambda^{(v+1)u_{1}} - 1] = \Lambda_{\lambda}(u_{1})$, $A \cap [0,\lambda^{(v+1)u_{1}} - 1] = A(u_{1},v+1)$ with $\lambda^{(v+1)u_{1}}\in A\cup \Lambda$, but $\lambda^{(v+1)u_{1}}\notin A\cap \Lambda$.

Assume that $\Lambda \cap [0,\lambda^{(v+1)u_{1}} - 1] \neq \Lambda_{\lambda}(u_{1})$ or $A \cap [0,\lambda^{(v+1)u_{1}} - 1] \neq A_{\lambda}(u_{1},v+1)$. Then there exists an integer $\lambda^{vu_{1}} < n < \lambda^{(v+1)u_{1}}$ such that $\Lambda \cap [0,n-1] = \Lambda_{\lambda}(u_{1})$,  $A \cap [0,n - 1] = A_{\lambda}(u_{1},v+1)\cap [0,n-1]$, but $\Lambda \cap [0,n] \neq \Lambda_{\lambda}(u_{1})$ or $A \cap [0,n] \neq A_{\lambda}(u_{1},v+1)\cap [0,n]$. These statements can only hold in the following four cases.

\textbf{Case 1.} $n\notin A_{\lambda}(u_{1},v+1)$, but $n\in A$. Then
in the following representation
\[
n = \lambda_{1}a^{(1)} + \dots{} + \lambda_{u_{1}}a^{(u_{1})},
\]
with $a^{(i)}\in  A_{\lambda}(u_{1},v+1)$, we have $a^{(i)} < n$ and so $a^{(i)}\in A$. Since $n\in A$, we have $n = \lambda_{1}\cdot n$, In view of Lemma 2.2 with $S = A$ and $\underline{\mu} = \underline{\lambda}$, this is a contradiction.

\textbf{Case 2.} $n\notin A_{\lambda}(u_{1},v+1)$ but $n\in \Lambda$. Then
in the following representation
\[
n = \lambda_{1}a^{(1)} + \dots{} + \lambda_{u_{1}}a^{(u_{1})},
\]
with $a^{(i)}\in  A_{\lambda}(u_{1},v+1)$, we have $a^{(i)} < n$ and so $a^{(i)}\in A$, for every $1 \le i \le u_{1}$. On the other hand, if $n = \lambda_{u_{1}+1}$, then $n = \lambda_{u_{1}+1}\cdot a_{2}$. In view of Lemma 2.2 with $S = A$ and $\underline{\mu} = \underline{\lambda}$, this is a contradiction.

\textbf{Case 3.} $n\in A_{\lambda}(u_{1},v+1)$ but $n\notin \Lambda \cup A$. Then
in the following representation
\[
n = \lambda_{1}a^{(1)} + \dots{} + \lambda_{u_{1}}a^{(u_{1})},
\]
with $a^{(i)}\in  A$,  we have $a^{(i)} < n$ and so $a^{(i)}\in A_{\lambda}(u_{1},v+1)$, for every $1 \le i \le u_{1}$. Moreover, $n = \lambda_{1}\cdot n$, $n\in A_{\lambda}(u_{1},v+1)$. In view of Lemma 2.1 and Lemma 2.2 with $S = A_{\lambda}(u_{1},v+1)$ and $\underline{\mu} = \Lambda(u_{1})$, this is a contradiction.

\textbf{Case 4.} $n\in A_{\lambda}(u_{1},v+1)$ and $n\in \Lambda$. Then $n = j_{0}\lambda^{vu_{1}}$, with $1 < j_{0} < \lambda$ and $\lambda_{u_{1}+1} = n$. Since
\[
0 \le \lambda^{vu_{1}+1} - \left \lfloor \frac{\lambda^{vu_{1}+1}}{n} \right  \rfloor  \cdot n < n, \hspace*{15mm}  0 < \left \lfloor \frac{\lambda^{vu_{1}+1}}{n} \right  \rfloor < \lambda,
\]
thus we have 
\[
\lambda^{vu_{1}+1} - \left \lfloor \frac{\lambda^{vu_{1}+1}}{n} \right  \rfloor \cdot n = \lambda_{1}a^{(1)} + \dots{} + \lambda_{u_{1}}a^{(u_{1})},
\]
with $a^{(1)}, \dots{} ,a^{(u_{1})}\in  A$, and so
\[
\lambda^{vu_{1}+1} = \lambda_{1}a^{(1)} + \dots{} + \lambda_{u_{1}}a^{(u_{1})} + \lambda_{u_{1}+1}\cdot a_{\left \lfloor \frac{\lambda^{vu_{1}+1}}{n} \right  \rfloor + 1}.
\]
On the other hand, $\lambda^{vu_{1}}\in A$, hence $\lambda^{vu_{1}+1} = \lambda_{2}\cdot \lambda^{vu_{1}}$.  
In view of Lemma 2.2 with $S = A$ and $\underline{\mu} = \underline{\lambda}$, this is a contradiction.

It follows that  $A \cap [0,\lambda^{(v+1)u_{1}} - 1]= A_{\lambda}(u_{1},v+1)$ and $\Lambda \cap [0,\lambda^{vu_{1}} - 1] = \Lambda(u_{1})$. Now we prove that $\lambda^{(v+1)u_{1}}\in \Lambda \cup A$, but
$\lambda^{(v+1)u_{1}}\notin \Lambda \cap A$. By (ii) of Lemma 2.1, if $0 \le n < \lambda^{(v+1)u_{1}}$, then
$R_{A_{\lambda}(u_{1},v+1), \Lambda(u_{1})}(n) = 1$ and if  $n \ge \lambda^{(v+1)u_{1}}$, then $R_{A_{\lambda}(u_{1},v+1), \Lambda(u_{1})}(n) = 0$. Let $A_{\lambda}(u_{1},v+1) = \{a_{1}, \dots{} ,a_{f}\}$, $a_{1} < a_{2} \dots{} < a_{f}$. Then the smallest element of the following set
\[
\{\lambda_{1}a^{(1)} + \dots{} + \lambda_{m}a^{(m)}: a^{(i)}\in  A\} \setminus \{\lambda_{1}a^{(1)} + \dots{} + \lambda_{u_{1}}a^{(u_{1})}: a^{(i)}\in A_{\lambda}(u_{1},v+1)\}
\]
is $\lambda^{(v+1)u_{1}}$ and also $\lambda_{1}a_{f+1} = a_{f+1}$ or $\lambda_{u_{1}+1}\cdot a_{2} = \lambda_{u_{1}+1}$, thus $\lambda^{(v+1)u_{1}}\in \Lambda \cup A$. Suppose that $\lambda^{(v+1)u_{1}}\in \Lambda \cap A$. Then we have  $\lambda_{u_{1}+1} = \lambda^{(v+1)u_{1}}$, and so $\lambda^{(v+1)u_{1}} = \lambda_{1}\cdot \lambda^{(v+1)u_{1}} = \lambda_{u_{1}+1} \cdot a_{2}$.
In view of Lemma 2.2 with $S = A$ and $\underline{\mu} = \underline{\lambda}$, this is a contradiction.

If $\lambda^{vu_{1}}\in A$ for every $v\in \mathbb{Z}^{+}$, then $A = A_{\lambda}(u_{1},\infty)$ and 
$\Lambda = \Lambda(u_{1})$, and the proof is completed. Otherwise, there exists an integer $v_{1} \ge 2$ such that for every $1 \le v < v_{1}$ we have $\lambda^{vu_{1}}\in A$ and $\lambda^{v_{1}u_{1}}\in \Lambda$. 
Then we have $A \cap [0,\lambda^{v_{1}u_{1}} - 1] = A_{\lambda}(u_{1},v_{1})$, $\Lambda \cap [0,\lambda^{v_{1}u_{1}} - 1] = \Lambda(u_{1})$, $\lambda^{v_{1}u_{1}}\in \Lambda$ and $\lambda^{v_{1}u_{1}}\notin A$.

To prove Theorem 1.1, we need the following claims.

\textbf{Claim 1.} Assume that the integers $u_{1}, v_{1}, u_{2}, v_{2}, \dots{} ,u_{h}, v_{h} \ge 2$ and $u \ge 1$ satisfy 
\[
\Lambda \cap [0,\lambda^{uU_{h}V_{h}} - 1] = \Lambda_{\lambda}(u_{1}, v_{1}, u_{2}, v_{2}, \dots{} ,u_{h}, v_{h}, u),
\]
\[
A \cap [0,\lambda^{uU_{h}V_{h}} - 1] = A_{\lambda}(u_{1}, v_{1}, u_{2}, v_{2}, \dots{} ,u_{h}, v_{h}) 
\]
and $\lambda^{uU_{h}V_{h}} \in \Lambda$, but $\lambda^{uU_{h}V_{h}} \notin A$.
Then we have 
\[
\Lambda \cap [0,\lambda^{(u+1)U_{h}V_{h}} - 1] = \Lambda_{\lambda}(u_{1}, v_{1}, u_{2}, v_{2}, \dots{} ,u_{h}, v_{h}, u+1), 
\]
\[
A \cap [0,\lambda^{(u+1)U_{h}V_{h}} - 1] = A_{\lambda}(u_{1}, v_{1}, u_{2}, v_{2}, \dots{} ,u_{h}, v_{h})
\]
and $\lambda^{(u+1)U_{h}V_{h}} \in A \cup \Lambda$, but $\lambda^{(u+1)U_{h}V_{h}} \notin A\cap \Lambda$.

\textbf{Claim 2.} Assume that the integers $u_{1}, v_{1}, u_{2}, v_{2}, \dots{} ,u_{h}, v_{h}, u_{h+1} \ge 2$ and $v \ge 1$ satisfy 
\[
\Lambda \cap [0,\lambda^{U_{h+1}V_{h}} - 1] = \Lambda_{\lambda}(u_{1}, v_{1}, u_{2}, v_{2}, \dots{} ,u_{h}, v_{h}, u_{h+1}),
\]
\[
A \cap [0,\lambda^{U_{h+1}V_{h}} - 1] = A_{\lambda}(u_{1}, v_{1}, u_{2}, v_{2}, \dots{} ,u_{h}, v_{h}, u_{h+1}, v)
\]
and $\lambda^{vU_{h+1}V_{h}} \in A$, but $\lambda^{vU_{h+1}V_{h}} \notin \Lambda$.
Then we have 
\[
\Lambda \cap [0,\lambda^{(v+1)U_{h+1}V_{h}} - 1] = \Lambda_{\lambda}(u_{1}, v_{1}, u_{2}, v_{2}, \dots{} ,u_{h}, v_{h}, u_{h+1}), 
\]
\[
A \cap [0,\lambda^{(v+1)U_{h+1}V_{h}} - 1] = A_{\lambda}(u_{1}, v_{1}, u_{2}, v_{2}, \dots{} ,u_{h}, v_{h}, u_{h+1}, v+1) 
\]
and $\lambda^{(v+1)U_{h+1}V_{h}} \in A \cup \Lambda$, but $\lambda^{(v+1)U_{h+1}V_{h}} \notin A\cap \Lambda$.

In the next step, we show how Theorem 1.1 follows from Claim 1 and Claim 2. We can assume that there exist integers $u_{1}, v_{1} \ge 2$ such that $A \cap [0,\lambda^{1\cdot U_{1}V_{1}} - 1] = A_{\lambda}(u_{1}, v_{1})$, $\Lambda \cap [0,\lambda^{1\cdot U_{1}V_{1}} - 1] = \Lambda_{\lambda}(u_{1})$ and $\lambda^{U_{1}V_{1}} \in \Lambda$, but $\lambda^{U_{1}V_{1}} \notin A$. Then by Claim 1, there exists an integer $u_{2}\ge 2$ such that for every $1\le u < u_{2}$ we have $\lambda^{uU_{1}V_{1}} \in \Lambda$, but $\lambda^{u_{2}U_{1}V_{1}} \in A$. Then by Claim 2, if $\Lambda \cap [0,\lambda^{vU_{2}V_{1}} - 1] = \Lambda_{\lambda}(u_{1},v_{1},u_{2})$, $A \cap [0,\lambda^{vU_{2}V_{1}} - 1] = A_{\lambda}(u_{1}, v_{1}, u_{2}, v)$ and $\lambda^{vU_{2}V_{1}} \notin \Lambda$, but $\lambda^{vU_{2}V_{1}} \in A$ holds for every $v\in \mathbb{Z}^{+}$, then $\Lambda = \Lambda_{\lambda}(u_{1},v_{1},u_{2})$ and $A = A_{\lambda}(u_{1}, v_{1}, u_{2}, \infty)$ and Theorem 1.1 follows. Otherwise,
there exists an integer $v_{2} \ge 2$ such that $\Lambda \cap [0,\lambda^{U_{2}V_{2}} - 1] = \Lambda_{\lambda}(u_{1},v_{1}, u_{2})$, $A \cap [0,\lambda^{U_{2}V_{2}} - 1] = A_{\lambda}(u_{1}, v_{1}, u_{2})$ and $\lambda^{U_{2}V_{2}} \in \Lambda$, but $\lambda^{U_{2}V_{2}} \notin A$. Continuing this process one can get a $k \in \mathbb{Z}^{+}$ with $\Lambda \cap [0,\lambda^{U_{k-1}V_{k-1}}] = \Lambda_{\lambda}(u_{1},v_{1}, \dots{} ,u_{k-2},v_{k-2}, u_{k-1})$, $A \cap [0,\lambda^{U_{k-1}V_{k-1}}] = A_{\lambda}(u_{1}, v_{1}, \dots{} ,u_{k-2},v_{k-2},u_{k-1},v_{k-1})$ and $\lambda^{U_{k-1}V_{k-1}} \in \Lambda$, but $\lambda^{U_{k-1}V_{k-1}} \notin A$. Moreover, by Claim 1, there exists an integer $u_{k}\ge 2$ such that $\Lambda \cap [0,\lambda^{U_{k}V_{k-1}} - 1] = \Lambda_{\lambda}(u_{1},v_{1}, \dots{} ,u_{k-2},v_{k-2}, u_{k-1}, v_{k-1}, u_{k})$, $A \cap [0,\lambda^{u_{k}U_{k-1}V_{k-1}} - 1] = A_{\lambda}(u_{1}, v_{1}, \dots{} ,u_{k-1},v_{k-1})$ and $\lambda^{U_{k}V_{k-1}} \notin \Lambda$, but $\lambda^{U_{k}V_{k-1}} \in A$ and 
\[
\Lambda \cap [0,\lambda^{vU_{k}V_{k-1}} - 1] = \Lambda_{\lambda}(u_{1},v_{1}, \dots{} ,u_{k-2},v_{k-2}, u_{k-1}, v_{k-1}, u_{k}), 
\]
$A \cap [0,\lambda^{vU_{k}V_{k-1}} - 1] = A_{\lambda}(u_{1}, v_{1}, \dots{} ,u_{k-1},v_{k-1},u_{k}, v)$ for every $v \in \mathbb{Z}^{+}$. Then we have
\[
\Lambda = \Lambda_{\lambda}(u_{1},v_{1}, \dots{} ,u_{k-2},v_{k-2}, u_{k-1}, v_{k-1}, u_{k}) 
\]
and  
\[
A = A_{\lambda}(u_{1}, v_{1}, \dots{} ,u_{k-1},v_{k-1}, u_{k}, \infty)
\]
and Theorem 1.1 follows.

\textbf{Proof of Claim 1.} We prove by contradiction. Suppose that there exists an integer $\lambda^{uU_{h}V_{h}} < n < \lambda^{(u+1)U_{h}V_{h}}$ such that 
\[
A \cap [0, n - 1] = A_{\lambda}(u_{1}, v_{1}, u_{2}, v_{2}, \dots{} ,u_{h}, v_{h})\cap [0, n - 1] 
\]
and
\[
\Lambda \cap [0, n - 1] = \Lambda_{\lambda}(u_{1}, v_{1}, u_{2}, v_{2}, \dots{} ,u_{h}, v_{h}, u+1)\cap [0, n - 1] 
\]
but 
\[
A \cap [0, n] \neq A_{\lambda}(u_{1}, v_{1}, u_{2}, v_{2}, \dots{} ,u_{h}, v_{h})\cap [0, n] 
\]
or
\[
\Lambda \cap [0, n] \neq \Lambda_{\lambda}(u_{1}, v_{1}, u_{2}, v_{2}, \dots{} ,u_{h}, v_{h}, u+1)\cap [0, n]. 
\]

Now let $\Lambda_{\lambda}(u_{1}, v_{1}, u_{2}, v_{2}, \dots{} ,u_{h}, v_{h}, u+1)\cap [0, n] = \{\lambda_{1}^{'}, \dots{} ,\lambda_{w}^{'}\}$, $\lambda_{1}^{'} < \dots{} < \lambda_{w}^{'}$. We have the following four cases.

\textbf{Case 1.} $n\notin \Lambda_{\lambda}(u_{1}, v_{1}, u_{2}, v_{2}, \dots{} ,u_{h}, v_{h}, u+1)$, but $n\in \Lambda$. Then $\lambda_{1}^{'}, \dots{} ,\lambda_{w}^{'}\in \Lambda$ and 
\[
n = \lambda_{1}^{'}a^{(1)} + \dots{} + \lambda_{w}^{'}a^{(w)},
\]
with $a^{(i)}\in  A_{\lambda}(u_{1}, v_{1}, u_{2}, v_{2}, \dots{} ,u_{h}, v_{h})$ and so $a^{(i)} < n$, $a^{(i)}\in  A$, for every $1\le i \le w$. On the other hand, $n = n\cdot a_{2}$, $n\in \Lambda$. In view of Lemma 2.2 with $S = A$ and $\underline{\mu} = \underline{\lambda}$, this is a contradiction. 

\textbf{Case 2.} $n\notin \Lambda_{\lambda}(u_{1}, v_{1}, u_{2}, v_{2}, \dots{} ,u_{h}, v_{h}, u+1)$, but $n\in A$. Then $\lambda_{1}^{'}, \dots{} ,\lambda_{w}^{'}\in \Lambda$ and 
\[
n = \lambda_{1}^{'}a^{(1)} + \dots{} + \lambda_{w}^{'}a^{(w)},
\]
and $a^{(i)}\in  A_{\lambda}(u_{1}, v_{1}, u_{2}, v_{2}, \dots{} ,u_{h}, v_{h})$ and so $a^{(i)} < n$ , $a^{(i)}\in  A$ for every $1\le i \le w$. On the other hand, $n = \lambda_{1}\cdot n$, $n\in A$. 
In view of Lemma 2.2 with $S = A$ and $\underline{\mu} = \underline{\lambda}$, this is a contradiction.

\textbf{Case 3.} $n\in \Lambda_{\lambda}(u_{1}, v_{1}, u_{2}, v_{2}, \dots{} ,u_{h}, v_{h}, u+1)$, but $n\notin A\cup \Lambda$. Then $\Lambda\cap [0, n] = \{\lambda_{1}^{'}, \dots{} ,\lambda_{w-1}^{'}\}$, and so 
\[
n = \lambda_{1}^{'}a^{(1)} + \dots{} + \lambda_{w-1}^{'}a^{(w-1)},
\]
where $a^{(i)}\in  A$,  $a^{(i)} < n$ and so $a^{(i)}\in  A_{\lambda}(u_{1}, v_{1}, u_{2}, v_{2}, \dots{} ,u_{h}, v_{h})$ for every $1\le i \le w - 1$. Furthermore, $n = n\cdot a_{2}$, where $n\in \Lambda$. In view of Lemma 2.1 and Lemma 2.2 with $S = A_{\lambda}(u_{1}, v_{1}, u_{2}, v_{2}, \dots{} ,u_{h}, v_{h})$ and
$\underline{\mu} = \Lambda_{\lambda}(u_{1}, v_{1}, u_{2}, v_{2}, \dots{} ,u_{h}, v_{h}, u+1)$, this is a  contradiction.

\textbf{Case 4.} 
 $n\in \Lambda_{\lambda}(u_{1}, v_{1}, u_{2}, v_{2}, \dots{} ,u_{h}, v_{h}, u+1)$, but $n\in A$. Let
 \[
 n = \lambda^{i_{p}U_{p}V_{p} + i_{p+1}U_{p+1}V_{p+1} + \dots{} + i_{h-1}U_{h-1}V_{h-1} + uU_{h}V_{h}},
 \]
where $p < h$ and $0 < i_{p} < u_{p+1}$. Then by the definition of $\Lambda_{\lambda}(u_{1}, v_{1}, u_{2}, v_{2}, \dots{} ,u_{h}, v_{h}, u+1)$, it follows that $\lambda^{(u_{p+1}-i_{p})U_{p}V_{p}}\in \Lambda_{\lambda}(u_{1}, v_{1}, u_{2}, v_{2}, \dots{} ,u_{h}, v_{h}, u+1)$, but since $\lambda^{(u_{p+1}-i_{p})U_{p}V_{p}} < n$, we have $\lambda^{(u_{p+1}-i_{p})U_{p}V_{p}}\in \Lambda$.
Moreover, by the definition of $A_{\lambda}(u_{1}, v_{1}, u_{2}, v_{2}, \dots{} ,u_{h}, v_{h})$ it follows that $\lambda^{u_{p+1}U_{p}V_{p}} \in A_{\lambda}(u_{1}, v_{1}, u_{2}, v_{2}, \dots{} ,u_{h}, v_{h})$, but since $\lambda^{u_{p+1}U_{p}V_{p}} < n$, we have $\lambda^{u_{p+1}U_{p}V_{p}}\in A$. By the definition of $\Lambda_{\lambda}(u_{1}, v_{1}, u_{2}, v_{2}, \dots{} ,u_{h}, v_{h}, u+1)$,
\[
\lambda^{i_{p+1}U_{p+1}V_{p+1} + \dots{} + i_{h-1}U_{h-1}V_{h-1} + uU_{h}V_{h}} \in \Lambda_{\lambda}(u_{1}, v_{1}, u_{2}, v_{2}, \dots{} ,u_{h}, v_{h}, u+1),
\]
but since $\lambda^{i_{p+1}U_{p+1}V_{p+1} + \dots{} + i_{h-1}U_{h-1}V_{h-1} + uU_{h}V_{h}} < n$, we have 
\[
\lambda^{i_{p+1}U_{p+1}V_{p+1} + \dots{} + i_{h-1}U_{h-1}V_{h-1} + uU_{h}V_{h}}\in \Lambda. 
\]
Then we have
\[
\lambda^{u_{p+1}U_{p}V_{p} + i_{p+1}U_{p+1}V_{p+1} + \dots{} + i_{h-1}U_{h-1}V_{h-1} + uU_{h}V_{h}} = \lambda^{(u_{p+1}-i_{p})U_{p}V_{p}}\cdot n
\]
\[
= \lambda^{i_{p+1}U_{p+1}V_{p+1} + \dots{} + i_{h-1}U_{h-1}V_{h-1} + uU_{h}V_{h}}\cdot \lambda^{u_{p+1}U_{p}V_{p}},
\]
In view of Lemma 2.2 with $S = A$ and $\underline{\mu} = \underline{\lambda}$, this is a contradiction.
 
 Finally, we have to show that $\lambda^{(u+1)U_{h}V_{h}}\in A \cup \Lambda$, but $\lambda^{(u+1)U_{h}V_{h}}\notin A \cap \Lambda$. We know that $A\cap [0,\lambda^{(u+1)U_{h}V_{h}}-1] = A_{\lambda}(u_{1}, v_{1}, u_{2}, v_{2}, \dots{} ,u_{h}, v_{h})$ and  $\Lambda\cap [0,\lambda^{(u+1)U_{h}V_{h}}-1] = \Lambda_{\lambda}(u_{1}, v_{1}, u_{2}, v_{2}, \dots{} ,u_{h}, v_{h}, u+1)$. Let $\Lambda_{\lambda}(u_{1}, v_{1}, u_{2}, v_{2}, \dots{} ,u_{h}, v_{h}, u+1) = \{\lambda_{1}, \dots{} ,\lambda_{z}\}$
 and $A_{\lambda}(u_{1}, v_{1}, u_{2}, v_{2}, \dots{} ,u_{h}, v_{h}) = \{a_{1}, \dots{} ,a_{t}\}$. Consider the following set
\begin{equation}
\{\lambda_{1}a^{(1)} + \dots{} + \lambda_{m}a^{(m)}: a^{(i)}\in  A\} \setminus 
\end{equation}
\[
\{\lambda_{1}a^{(1)} + \dots{} + \lambda_{z}a^{(z)}: a^{(i)}\in A_{\lambda}(u_{1}, v_{1}, u_{2}, v_{2}, \dots{} ,u_{h}, v_{h})\}.
\]
By Lemma 2.1, we have for every $0 \le n < \lambda^{(u+1)U_{h}V_{h}}$ can be written in the form 
\[
n = \lambda_{1}a^{(1)} + \dots{} + \lambda_{z}a^{(z)},
\]
where $a^{(i)}\in A_{\lambda}(u_{1}, v_{1}, u_{2}, v_{2}, \dots{} ,u_{h}, v_{h})$.
Moreover, if $n \ge \lambda^{(u+1)U_{h}V_{h}}$, then 
\[
n \neq \lambda_{1}a^{(1)} + \dots{} + \lambda_{z}a^{(z)},
\]
where $a^{(i)}\in A_{\lambda}(u_{1}, v_{1}, u_{2}, v_{2}, \dots{} ,u_{h}, v_{h})$.
This implies that the smallest element of the set in (1) is $\lambda^{(u+1)U_{h}V_{h}}$.
On the other hand, the smallest element of the set in (1) is either $\lambda_{1}a_{t+1} = a_{t+1}$ or  
$\lambda_{z+1}a_{2} = \lambda_{z+1}$. This implies that $\lambda^{(u+1)U_{h}V_{h}} \in A\cup \Lambda$.
If $\lambda^{(u+1)U_{h}V_{h}} \in A\cap \Lambda$, then in the representation $\lambda^{(u+1)U_{h}V_{h}} = \lambda_{1}\cdot  \lambda^{(u+1)U_{h}V_{h}} = \lambda^{(u+1)U_{h}V_{h}}\cdot a_{2}$.
In view of Lemma 2.2 with $S = A$ and $\underline{\mu} = \underline{\lambda}$, this is a contradiction.

\textbf{Proof of Claim 2.} 
We prove by contradiction. Suppose that there exists an integer $\lambda^{vU_{h+1}V_{h}} < n < \lambda^{(v+1)U_{h+1}V_{h}}$ such that 
\[
A \cap [0, n - 1] = A_{\lambda}(u_{1}, v_{1}, u_{2}, v_{2}, \dots{} ,u_{h}, v_{h}, u_{h+1}, v+1)\cap [0, n - 1], 
\]
\[
\Lambda \cap [0, n - 1] = \Lambda_{\lambda}(u_{1}, v_{1}, u_{2}, v_{2}, \dots{} ,u_{h}, v_{h}, u_{h+1})\cap [0, n - 1] 
\]
but 
\[
A \cap [0, n] \neq A_{\lambda}(u_{1}, v_{1}, u_{2}, v_{2}, \dots{} ,u_{h}, v_{h},  u_{h+1}, v+1)\cap [0, n], 
\]
or
\[
\Lambda \cap [0, n] \neq \Lambda_{\lambda}(u_{1}, v_{1}, u_{2}, v_{2}, \dots{} ,u_{h}, v_{h}, u_{h+1})\cap [0, n]. 
\]
Now let $\Lambda_{\lambda}(u_{1}, v_{1}, u_{2}, v_{2}, \dots{} ,u_{h}, v_{h}, u_{h+1})\cap [0, n] = \{\lambda_{1}, \dots{} ,\lambda_{d}\}$, $\lambda_{1} < \dots{} < \lambda_{d}$. We have the following four cases.

\textbf{Case 1.} $n\notin A_{\lambda}(u_{1}, v_{1}, u_{2}, v_{2}, \dots{} ,u_{h}, v_{h}, u_{h+1})$, but $n\in A$. Then $\lambda_{d} < n$. Hence,  $\lambda_{1}, \dots{} ,\lambda_{d}\in \Lambda$ and consider the representation
\[
n = \lambda_{1}a^{(1)} + \dots{} + \lambda_{d}a^{(d)}
\]
with $a^{(1)}, \dots{} ,a^{(d)}\in  A_{\lambda}(u_{1}, v_{1}, u_{2}, v_{2}, \dots{} ,u_{h}, v_{h}, u_{h+1}, v+1)$. Since $a^{(i)} < n$, we have $a^{(i)}\in  A$ for every $1 \le i \le d$. On the other hand, $n = \lambda_{1}\cdot n$, $n\in A$. In view of Lemma 2.2 with $S = A$ and $\underline{\mu} = \underline{\lambda}$, this is a contradiction.

\textbf{Case 2.} $n\notin A_{\lambda}(u_{1}, v_{1}, u_{2}, v_{2}, \dots{} ,u_{h}, v_{h}, u_{h+1})$, but $n\in \Lambda$. Then $\lambda_{d} < n$. Hence, $\lambda_{1}, \dots{} ,\lambda_{d}\in \Lambda$ and consider the representation
\[
n = \lambda_{1}a^{(1)} + \dots{} + \lambda_{d}a^{(d)}
\]
with $a^{(1)}, \dots{} ,a^{(d)}\in A_{\lambda}(u_{1}, v_{1}, u_{2}, v_{2}, \dots{} ,u_{h}, v_{h}, u_{h+1}, v+1)$. Since $a^{(i)} < n$, we have 
$a^{(1)}, \dots{} ,a^{(d)}\in  A$. On the other hand, $n = \lambda_{d+1} = \lambda_{d+1}\cdot a_{2}$.
In view of Lemma 2.2 with $S = A$ and $\underline{\mu} = \underline{\lambda}$, this is a contradiction.

\textbf{Case 3.} $n\in A_{\lambda}(u_{1}, v_{1}, u_{2}, v_{2}, \dots{} ,u_{h}, v_{h}, u_{h+1}, v+1)$, but $n\notin A\cup \Lambda$. Then in the representation 
$n = \lambda_{1}a^{(1)} + \dots{} + \lambda_{m}a^{(m)}$, $a^{(i)}\in  A$, we have $a^{(i)} = 0$ for $i > d$. Hence $n = \lambda_{1}a^{(1)} + \dots{} + \lambda_{d}a^{(d)}$,  $a^{(i)} < n$ and so $a^{(i)}\in  A_{\lambda}(u_{1}, v_{1}, u_{2}, v_{2}, \dots{} ,u_{h}, v_{h}, u_{h+1}, v+1)$ for every $1\le i \le d$. Furthermore, $n = \lambda_{1}\cdot n$, where $n\in A_{\lambda}(u_{1}, v_{1}, u_{2}, v_{2}, \dots{} ,u_{h}, v_{h},u_{h+1}, v+1)$. In view of Lemma 2.1 and Lemma 2.2 with $S = A_{\lambda}(u_{1}, v_{1}, u_{2}, v_{2}, \dots{} ,u_{h}, v_{h}, u_{h+1}, v+1)$ and
$\underline{\mu} = \Lambda_{\lambda}(u_{1}, v_{1}, u_{2}, v_{2}, \dots{} ,u_{h}, v_{h}, u_{h+1})$, this is a  contradiction.

\textbf{Case 4.}
$n\in A_{\lambda}(u_{1}, v_{1}, u_{2}, v_{2}, \dots{} ,u_{h}, v_{h}, u_{h+1}, v+1)$, and $n\in \Lambda$.
Then $n = \lambda_{d+1}$, $\lambda^{vU_{h+1}V_{h}} < n < \lambda^{vU_{h+1}V_{h}+1}$. Let 
\[
n = \sum_{i=0}^{vU_{h+1}V_{h}}\delta_{i}\lambda^{i}, 
\]
$0 \le \delta_{i} < \lambda$, $\delta_{vU_{h+1}V_{h}} \neq 0$. Then we have
\[
0 \le \lambda^{vU_{h+1}V_{h}+1} - \left \lfloor \frac{\lambda^{vU_{h+1}V_{h}+1}}{n} \right  \rfloor  \cdot n < n, 
\textnormal{ where } 1\le \left \lfloor \frac{\lambda^{vU_{h+1}V_{h}+1}}{n} \right  \rfloor < \lambda,
\]
thus we have 
\[
\lambda^{vU_{h+1}V_{h}+1} - \left \lfloor \frac{\lambda^{vU_{h+1}V_{h}+1}}{n} \right  \rfloor = \lambda_{1}a^{(1)} + \dots{} + \lambda_{d}a^{(d)},
\]
with $a^{(i)}\in  A$, and so
\[
\lambda^{vU_{h+1}V_{h}+1} = \lambda_{1}a^{(1)} + \dots{} + \lambda_{d}a^{(d)} + \lambda_{d+1}\cdot a_{\left \lfloor \frac{\lambda^{vU_{h+1}V_{h}+1}}{n} \right  \rfloor + 1}.
\]
On the other hand, it follows from the definition of $A_{\lambda}(u_{1}, v_{1}, u_{2}, v_{2}, \dots{} ,u_{h}, v_{h}, u_{h+1}, v+1)$ that $\lambda^{vU_{h+1}V_{h}} \in A_{\lambda}(u_{1}, v_{1}, u_{2}, v_{2}, \dots{} ,u_{h}, v_{h}, u_{h+1}, v+1)$. Since $\lambda^{vU_{h+1}V_{h}} < n$, and so $\lambda^{vU_{h+1}V_{h}}\in A$. Then we have $\lambda^{vU_{h+1}V_{h}+1} = \lambda_{2}\cdot \lambda^{vU_{h+1}V_{h}}$. In view of Lemma 2.2 with $S = A$ and $\underline{\mu} = \underline{\lambda}$, this is a contradiction. 

Finally, we have to show that $\lambda^{(v+1)U_{h+1}V_{h}}\in A \cup \Lambda$, but $\lambda^{(v+1)u_{h+1}U_{h}V_{h}}\notin A \cap \Lambda$. Let $\Lambda_{\lambda}(u_{1}, v_{1}, u_{2}, v_{2}, \dots{} ,u_{h}, v_{h}, u_{h+1}) = \{\lambda_{1}, \dots{} ,\lambda_{q}\}$, $\lambda_{1} < \dots{} < \lambda_{q}$. Consider the following set
\begin{equation}
\{\lambda_{1}a^{(1)} + \dots{} + \lambda_{m}a^{(m)}: a^{(i)}\in  A\} \setminus 
\end{equation}
\[
\{\lambda_{1}a^{(1)} + \dots{} + \lambda_{q}a^{(q)}: a^{(i)}\in A_{\lambda}(u_{1}, v_{1}, u_{2}, v_{2}, \dots{} ,u_{h}, v_{h}, u_{h+1}, v+1)\}.
\]
By Lemma 2.1, we have for every $0 \le n < \lambda^{(v+1)U_{h+1}V_{h}}$ can be written in the form 
\[
n = \lambda_{1}a^{(1)} + \dots{} + \lambda_{q}a^{(q)},
\]
where $a^{(i)}\in A_{\lambda}(u_{1}, v_{1}, u_{2}, v_{2}, \dots{} ,u_{h}, v_{h}, u_{h+1}, v+1)$ for every $1\le i \le q$.
Moreover, if $n \ge \lambda^{(v+1)U_{h+1}V_{h}}$ cannot be written in the form 
\[
n = \lambda_{1}a^{(1)} + \dots{} + \lambda_{q}a^{(q)},
\]
where $a^{(i)}\in A_{\lambda}(u_{1}, v_{1}, u_{2}, v_{2}, \dots{} ,u_{h}, v_{h}, u_{h+1}, v+1)$.
This implies that the smallest element of the set in (2) is $\lambda^{(v+1)U_{h+1}V_{h}}$.
Let $A_{\lambda}(u_{1}, v_{1}, u_{2}, v_{2}, \dots{} ,u_{h}, v_{h}, u_{h+1}, v+1) 
= \{a_{1}, \dots{} ,a_{l}\}$
Then obviously, the  the smallest element of the set in (2) is either $\lambda_{1}a_{l+1} = a_{l+1}$ or  
$\lambda_{d+1}a_{2} = \lambda_{d+1}$. This implies that $\lambda^{(v+1)U_{h+1}V_{h}} \in A\cup \Lambda$.
If $\lambda^{(v+1)U_{h+1}V_{h}} \in A\cap \Lambda$, then in the representation $\lambda^{(v+1)U_{h+1}V_{h}} = 1\cdot  \lambda^{(v+1)U_{h+1}V_{h}} = \lambda^{(v+1)U_{h+1}V_{h}}\cdot 1$ we have $1\in A\cap \Lambda$.
 In view of Lemma 2.2 with $S = A$ and $\underline{\mu} = \underline{\lambda}$, this is a contradiction.


\begin{thebibliography}{99}

\normalsize
\baselineskip=17pt

\bibitem{cr} J. Cilleruelo, J. Ru\'e, \emph{On a question of S\'ark\"ozy and S\'os for bilinear forms}, Bull. Lond. Math. Soc., 41 (2009), no. 2, 274-280.
\bibitem{gd} G. Dirac, \emph{Note on a problem in additive
number theory}, J. London Math. Soc., 1-26 (1951), no. 4, 312-313.
\bibitem{lm} L. Moser, \emph{An application of generating series}, Math. Mag., 35 (1962), no. 1, 37-38.
\bibitem{jr} J. Ru\'e, \emph{On polynomial representation functions for multivariate linear forms}, European J. Combin. 34 (2013), no. 8, 1429-1435.
\bibitem{rs} J. Ru\'e, C. Spiegel, \emph{On a problem of S\'ark\"ozy and S\'os for multivariate linear forms},
Rev. Mat. Iberoam. 36 (2020), no. 7, 2107-2119.
\bibitem{st} A. S\'ark\"ozy, V. T. S\'os, \emph{On additive representation functions}, In: The mathematics of Paul Erd\H{o}s I, 129-150. Algorithms Combin. 13, Springer, Berlin, 1997. 
\end{thebibliography}
\end{document}